\documentclass[reqno]{amsart}
\usepackage{amsmath}
\usepackage{amsfonts}
\usepackage{amssymb}
\usepackage{amstext}
\usepackage{amsbsy}
\usepackage{amsopn}
\usepackage{amsthm}
\usepackage{amsxtra}
\usepackage{upref}
\usepackage{graphicx}
\usepackage{epstopdf}
\DeclareFontFamily{OML}{rsfs}{\skewchar\font'177}
\DeclareFontShape{OML}{rsfs}{m}{n}{ <5> <6> rsfs5 <7> <8> <9> rsfs7
  <10> <10.95> <12> <14.4> <17.28> <20.74> <24.88> rsfs10 }{}
\DeclareMathAlphabet{\mathfs}{OML}{rsfs}{m}{n}

\newcommand{\x}{\ensuremath{\underline{x}}}
\newtheorem{thm}{Theorem}[section]
\newtheorem{lem}[thm]{Lemma}

\newtheorem{cor}[thm]{Corollary}
\newtheorem{defi}[thm]{Definition}
\newtheorem{theorem}{Theorem}[section]
\newtheorem*{theorem*}{Theorem}

\newtheorem*{example*}{Example}
\newtheorem{definition}{Definition}[section]

\numberwithin{equation}{section}

\renewcommand{\mod}{\mbox{$\,\mathrm{mod}\,$}}

\def\var{\textrm{var}}

\renewcommand{\epsilon}{\varepsilon}

\def\text#1{\textrm{#1}}

\def\emptyset{\varnothing}

\def\e{\epsilon}

\def\vf{\varphi}
\def\a{\alpha}
\def\b{\beta}

\def\d{\delta}

\def\g{\gamma}

\def\l{\lambda}

\def\s{\sigma}

\def\x{\times}
\def \R{\mathbb R}
\def \N{{\mathbb N}}

\def \Z{\mathbb Z}
\def \T{\mathbb T}

\def\un{\underline}

\def\wh{\widehat}
\def\({\biggl(}
\def\){\biggr)}

\def\<{\bold\langle}
\def\>{\bold\rangle}

\title[Bernoulli equilibrium states for surface diffeomorphisms]{Bernoulli equilibrium states for surface diffeomorphisms}
\author{Omri M. Sarig}\thanks{This work was supported by  ERC award  ERC-2009-StG n$^\circ$ 239885.}
\date{April 23, 2011}
\address{Faculty of Mathematics and Computer Science\\ The Weizmann Institute of Science\\ POB 26, Rehovot, Israel}
\email{omri.sarig@weizmann.ac.il}
\begin{document}
\maketitle
\begin{abstract}
Suppose $f:M\to M$ is a $C^{1+\a}$ $(\alpha>0)$ diffeomorphism on a compact smooth orientable manifold $M$ of dimension two, and let $\mu_\Psi$ be an equilibrium measure for a H\"older continuous potential $\Psi:M\to\R$. We show that if $\mu_\Psi$ has positive metric entropy, then $f$ is measure theoretically isomorphic mod $\mu_\Psi$ to the product of a Bernoulli scheme and a finite rotation.
\end{abstract}
\section{Statements}
Suppose $f:M\to M$ is a $C^{1+\a}$ ($\a>0$) diffeomorphism on a compact smooth orientable manifold $M$ of dimension two. Suppose $\Psi:M\to\R$ is H\"older continuous. An invariant probability measure $\mu$ is called an {\em equilibrium measure}, if it maximizes the quantity $h_\mu(f)+\int\Psi d\mu$, where $h_\mu(f)$ is the metric entropy. Such measures always exist when $f$ is $C^\infty$, because in this case the function $\mu\mapsto h_\mu(f)$ is upper semi--continuous \cite{N}. Let $\mu_\Psi$ be an ergodic equilibrium measure of $\Psi$. We prove:


\begin{thm}\label{ThmMain}
If $h_{\mu_\Psi}(f)>0$, then $f$ is   measure theoretically isomorphic with respect to $\mu_\Psi$ to the product of a Bernoulli scheme (see \S\ref{BernoulliSection}) and a finite rotation (a map of the form $x\mapsto (x+1)\mod p$ on $\{0,1,\ldots,p-1\}$).
\end{thm}
\noindent
%
The proof applies to other potentials, such as $-t\log J_u$ ($t\in\R$) where $J_u$ is the unstable Jacobian, see \S\ref{SectionJ_u}. In the particular case of the measure of maximal entropy ($\Psi\equiv 0$) we can say more, see \S\ref{MaxSubSection}.

The theorem  is false in higher dimension: Let $f$ denote the product of a hyperbolic toral automorphism and an irrational rotation. This $C^\infty$ diffeomorphism has many equilibrium measures of positive entropy. But $f$  cannot satisfy the conclusion of the theorem with respect to any of these measures, because $f$ has  the irrational rotation as a factor, and therefore none of its powers can have ergodic components with the $K$ property.

Bowen \cite{Bowen} and Ratner \cite{Ratner} proved Theorem \ref{ThmMain} for uniformly hyperbolic diffeomorphisms.
 In the non-uniformly hyperbolic case, Pesin proved that any absolutely continuous ergodic invariant measure all of whose  Lyapunov exponents are non-zero is isomorphic to the product of a Bernoulli scheme and a finite rotation \cite{Pesin}. By Pesin's Entropy Formula and Ruelle's Entropy Inequality, these measures are equilibrium measures of $-\log J_u$. Ledrappier extended Pesin's result to all equilibrium measures  with non-zero exponents for the potential $-\log J_u$, including those which are not absolutely continuous \cite{Ledrappier}.   These results hold in any dimension.

The work of Pesin and Ledrappier (see also \cite{OW}) uses the following property of equilibrium measures of  $-\log J_u$: the conditional measures on unstable manifolds are absolutely continuous \cite{Ledrappier}.    This is false for general H\"older potentials  \cite{LY}.

Theorem \ref{ThmMain} is proved in three steps:
\begin{enumerate}
\item {\bf Symbolic dynamics:}
 Any ergodic equilibrium measure on $M$ with positive entropy is a finite-to-one H\"older factor of an ergodic equilibrium measure on a countable Markov shift (CMS).
\item {\bf Ornstein Theory:} Factors of equilibrium measures of H\"older potentials on topologically mixing CMS are Bernoulli.
\item {\bf Spectral Decomposition:} The non-mixing case.
\end{enumerate}

\subsection*{Notation} $a=M^{\pm 1}b$ means $M^{-1}b\leq a\leq Mb$.

\section{Step One: Symbolic dynamics}\label{SymbolicSection}
Let $\mathfs G$ be a directed graph with a countable collection of vertices $\mathfs V$ s.t. every vertex has at least one  edge coming in, and at least one edge coming out. The  {\em countable Markov shift} (CMS) associated to $\mathfs G$ is the set
$$
\Sigma=\Sigma(\mathfs G):=\{(v_i)_{i\in\Z}\in\mathfs V^{\mathbb Z}:v_i\to v_{i+1}\textrm{ for all }i\}.
$$
The {\em natural metric} $d(\un{u},\un{v}):=\exp[-\min\{|i|:u_i\neq v_i\}]$ turns $\Sigma$ into a complete separable metric space. $\Sigma$ is compact iff $\mathfs G$ is finite. $\Sigma$  is locally compact iff every vertex of $\mathfs G$ has finite degree.
The {\em cylinder sets}
\begin{equation}\label{Cyl1}
_m[a_m,\ldots,a_n]:=\{(v_i)_{i\in\Z}\in\Sigma:v_i=a_i\ (i=m,\ldots,n)\}
\end{equation}
form a basis for the topology, and they generate the Borel $\s$--algebra $\mathfs B(\Sigma)$.

The {\em left shift map} $\s:\Sigma\to \Sigma$ is defined by $\s[(v_i)_{i\in\Z}]=(v_{i+1})_{i\in\Z}$.  Given $a,b\in\mathfs V$, write  $a\xrightarrow[]{n}b$ when there is a path $a\to v_1\to\cdots\to v_{n-1}\to b$ in $\mathfs G$. The left shift is topologically transitive iff  $\forall a,b\in\mathfs V\ \exists n\,(a\xrightarrow[]{n}b)$.  In this case $\gcd\{n:a\xrightarrow[]{n}a\}$ is the same for all $a\in\mathfs V$, and is called the {\em period} of $\s$. The left shift is topologically mixing iff it is topologically transitive and its period is equal to one. See \cite{Ki}.

Let
$
\Sigma^\#:=\{(v_i)_{i\in\Z}\in\Sigma:\exists u,v\in\mathfs V\exists n_k,m_k\uparrow\infty\textrm{ s.t. } v_{-m_k}=u, v_{n_k}=v\}.
$
Every $\s$--invariant probability measure gives $\Sigma^\#$ full measure, because of Poincar\'e's Recurrence Theorem.

Suppose $f:M\to M$ is a $C^{1+\a}$--diffeomorphism of a compact orientable smooth manifold $M$ s.t. $\dim M=2$. If $h_{top}(f)=0$, then every $f$--invariant measure has zero entropy by the variational principle, and Theorem \ref{ThmMain} holds trivially.  So we assume without loss of generality that $h_{top}(f)>0$. Fix $0<\chi<h_{top}(f)$.

A set $\Omega\subset M$ is called {\em $\chi$--large}, if $\mu(\Omega)=1$ for every ergodic invariant probability measure $\mu$ whose entropy is greater than $\chi$. The following theorems are  in \cite{S}:
\begin{thm}\label{Theorem_Main_Extension}
There exists a locally compact countable Markov shift $\Sigma_\chi$ and a H\"older continuous map $\pi_\chi:\Sigma_\chi\to M$ s.t.  $\pi_\chi\circ \s=f\circ \pi_\chi$,
 $\pi_\chi[\Sigma_\chi^\#]$ is $\chi$--large, and s.t.
every point in $\pi_\chi[\Sigma_\chi^\#]$ has finitely many pre-images.
\end{thm}
\begin{thm}\label{Theorem_Main_Finite_To_One}
Denote the set of states of $\Sigma_\chi$ by $\mathfs V_\chi$. There exists a function $\vf_\chi:\mathfs V_\chi\x \mathfs V_\chi\to\N$ s.t. if $x=\pi_\chi[(v_i)_{i\in\Z}]$ and $v_i=u$ for infinitely many negative $i$, and $v_i=v$ for infinitely many positive $i$, then $|\pi_\chi^{-1}(x)|\leq \vf_\chi(u,v)$.
\end{thm}
\begin{thm}\label{Theorem_Main_Lift}
 Every ergodic $f$--invariant probability measure $\mu$ on $M$ such that $h_\mu(f)>\chi$ equals
$\wh{\mu}\circ\pi_\chi^{-1}$ for some   ergodic $\s$--invariant probability measure $\wh{\mu}$ on $\Sigma_\chi$ with the same entropy.
\end{thm}

We will use these results to reduce the problem of Bernoullicity for equilibrium measures for  $f:M\to M$ and the potential $\Psi$, to the problem of  Bernoullicity for equilibrium measures for $\s:\Sigma_\chi\to\Sigma_\chi$ and the potential $\psi:=\Psi\circ\pi_\chi$.

\section{Step Two: Ornstein Theory}

First we  describe the structure of equilibrium measures of H\"older continuous potentials on countable Markov shifts (CMS), and then we show how this structure forces, in the topologically mixing case, isomorphism to a Bernoulli scheme.

\subsection{Equilibrium measures on one--sided CMS \cite{BS}} Suppose $\mathfs G$ is countable directed graph.
The {\em one-sided countable Markov shift} associated to $\mathfs G$ is
$$
\Sigma^+=\Sigma^+(\mathfs G):=\{(v_i)_{i\geq 0}\in\mathfs V^{\N\cup\{0\}}:v_i\to v_{i+1}\textrm{ for all $i$}\}.
$$
Proceeding as in the two--sided case,
we equip $\Sigma^+$ with the metric $d(\un{u},\un{v}):=\exp[-\min\{i\geq 0: u_i\neq v_i\}]$. The {\em cylinder sets}
\begin{equation}\label{Cyl2}
[a_0,\ldots,a_{n-1}]:=\{\un{u}\in\Sigma^+:u_i=a_i\ (i=0,\ldots,n-1)\}
\end{equation}
form a basis for the topology of $\Sigma^+$. Notice that unlike the two-sided case (\ref{Cyl1}), there is no left subscript: the cylinder starts at the zero coordinate.

The {\em left shift map} $\s:\Sigma^+\to\Sigma^+$ is given by $\s:(v_0,v_1,\ldots)\mapsto (v_1,v_2,\ldots)$. This map is not invertible. The natural extension of $(\Sigma^+,\s)$ is conjugate to $(\Sigma,\s)$.

A function $\phi:\Sigma^+\to\R$ is called {\em weakly H\"older continuous} if there are constants $C>0$ and $\theta\in(0,1)$ s.t. $\var_n\phi<C\theta^n$ for all $n\geq 2$, where
$$
\var_n\phi:=\sup\{\phi(\un{u})-\phi(\un{v}):u_i=v_i\ (i=0,\ldots,n-1)\}.
$$
The following inequality holds:
\begin{equation}\label{var}
\var_{n+m}\left(\sum_{j=0}^{n-1}\phi\circ\s^j\right)\leq \sum_{j=m+1}^\infty\var_j\phi.
\end{equation}
If  $\phi$ is bounded, weak H\"older continuity is the same as H\"older continuity.

The equilibrium measures for weakly H\"older potentials were described by Ruelle \cite{Rue} for finite graphs and  by Buzzi and the author for countable graphs \cite{BS}.
Make the following assumptions:
\begin{enumerate}
\item[(a)] $\s:\Sigma\to\Sigma$ is topologically mixing.
\item[(b)] $\phi$ is weakly H\"older continuous and $\sup\phi<\infty$. (This can be relaxed \cite{BS}.)
\item[(c)] $P_G(\phi):=\sup\{h_m(\s)+\int\phi dm\}<\infty$, where the supremum ranges over all shift invariant measures $m$ s.t. $h_m(\s)+\int\phi dm\neq \infty-\infty$. The potentials we will study satisfy $P_G(\phi)\leq h_{top}(f)+\max|\Psi|<\infty$.
\end{enumerate}
Define for $F:\Sigma^+\to\R^+$, $(L_{\phi}F)(\un{x})=\sum_{\s(\un{y})=\un{x}}e^{\phi(\un{y})}F(\un{y})$ (``Ruelle's operator"). The iterates of $L_\phi$ are  $(L_\phi^n F)(\un{x})=\sum_{\s^n(\un{y})=\un{x}}e^{\phi(\un{y})+\phi(\s\un{y})+\cdots+\phi(\s^{n-1}\un{y})}F(\un{y})$.

\begin{thm}[Buzzi \& S.]
Under assumptions {\em (a),(b),(c)} $\phi:\Sigma^+\to\R$ has at most one equilibrium measure. If this measure exists, then it is equal to $hd\nu$ where
\begin{enumerate}
\item  $h:\Sigma^+\to\R$ is a positive continuous function s.t. $L_{\phi}h=\l h$;
\item $\nu$ is a Borel measure on $\Sigma$ which is finite and positive on cylinder sets, $L_{\phi}^\ast\nu=\l\nu$, and $\int h d\nu=1$;
\item $\l=\exp P_G(\phi)$ and $\l^{-n}L_\phi^n 1_{[\un{a}]}\xrightarrow[n\to\infty]{}\nu[\un{a}]h$ pointwise for every cylinder $[\un{a}]$.
\end{enumerate}
Parts (1) and (2) continue to hold if we replace (a) by topological transitivity.
\end{thm}
\begin{cor}\label{CorProd}
Assume {\em (a),(b),(c)} and let $\mu$ be the equilibrium measure of $\phi$. For every finite $S^\ast\subset\mathfs V$ there exists a constant $C^\ast=C^\ast(S^\ast)>1$ s.t.  for every $m,n\geq 1$, every $n$--cylinder $[\un{a}]$, and every  $m$--cylinder $[\un{c}]$,
\begin{enumerate}
\item if the last symbol in $\un{a}$ is in $S^\ast$ and $[\un{a},\un{c}]\neq\emptyset$, then $1/C^\ast\leq
\frac{\mu{[}\un{a},\un{c}]}{\mu[\un{a}]\mu[\un{c}]}\leq C^\ast
$;
\item if the first symbol of $\un{a}$ is in $S^\ast$ and $[\un{c},\un{a}]\neq\emptyset$, then
$
1/C^\ast\leq \frac{\mu{[}\un{c},\un{a}]}{ \mu[\un{a}]\mu[\un{c}]}\leq C^\ast.
$
\end{enumerate}
\end{cor}
\begin{proof} We begin with a couple of observations (see \cite{Wg}).

\medskip
\noindent
{\em Observation 1.\/} Let
$
\phi^\ast:=\phi+\log h-\log h\circ\s-\log\l,
$
then $\phi^\ast$ is  weakly H\"older continuous, and if $L=L_{\phi^\ast}$ then $L^\ast\mu=\mu$, $L1=1$, and
$
L^n1_{[\un{a}]}\xrightarrow[n\to\infty]{}\mu[\un{a}]\textrm{ pointwise}.
$
Notice that $\phi^\ast$ need not be bounded.

\medskip
\noindent
{\em Proof.\/}
The convergence $\l^{-n}L_\phi^n 1_{[\un{a}]}\xrightarrow[n\to\infty]{}h \nu[\un{a}]$ and (\ref{var}) imply  that $\log h$ is weakly H\"older continuous, and $\var_1(\log h)<\infty$. It follows that  $\phi^\ast$ is weakly H\"older continuous.
The identities $L^\ast\mu=\mu, L1=1$ can be verified by direct calculation.
To see the convergence  $L^n 1_{[\un{a}]}\xrightarrow[n\to\infty]{}h \nu[\un{a}]$ we argue as follows.  Since $\phi$ has an equilibrium measure, $\phi$ is positive recurrent, see \cite{BS}. Positive recurrence is invariant under the addition of constants and coboundaries, so $\phi^\ast$ is also positive recurrent. The limit now follows from a theorem in \cite{SNR}.

\medskip
\noindent
{\em Observation 2\/}: For any positive continuous functions $F,G:\Sigma^+\to\R^+$,
\begin{equation}\label{transfer}
\int F (G\circ\s^n) d\mu=\int (L^n F)Gd\mu.
\end{equation}

\medskip
\noindent
{\em Proof.\/} Integrate the identity  $(L^n F)G=L^n(F G\circ\s^n)$ using  $L^\ast\mu=\mu$.

\medskip
\noindent
{\em Observation 3\/}: Let $\phi^\ast_n:=\phi^\ast+\phi^\ast\circ\s+\cdots+\phi^\ast\circ\s^{n-1}$, then  $M:=\exp(\sup_{n\geq 1}\var_{n+1}\phi_n^\ast)$ is finite, because of (\ref{var}) and the weak H\"older continuity of $\phi^\ast$.

\medskip
We turn to the proof of the corollary.
Suppose $\un{a}=(a_0,\ldots,a_{n-1})$ and $a_{n-1}\in S^\ast$. By  Observation 2,
$
\mu({[}\un{a},\un{c}])=\int_{[\un{c}]} L^n 1_{[\un{a}]}d\mu=\int e^{\phi_n^\ast(\un{a},y)}1_{[\un{c}]}(y)d\mu(y).
$

 It holds that
$
e^{\phi_n^\ast(\un{a},y)}=M^{\pm 1}e^{\phi_n^\ast(\un{a},z)}\textrm{ for all }y,z\in \s[a_{n-1}]$.
Fixing $y$ and averaging over $z\in \s[a_{n-1}]$ we obtain
\begin{align*}
e^{\phi_n^\ast(\un{a},y)}&= M^{\pm 1}\left(\frac{1}{\mu(\s[a_{n-1}])}\int_{\s[a_{n-1}]}e^{\phi_n^\ast(\un{a},z)}d\mu(z)\right)\\
&= M^{\pm 1}\left(\frac{1}{\mu(\s[a_{n-1}])}\int L^n 1_{[\un{a}]} d\mu\right)
=M^{\pm 1}\left(\frac{\mu[\un{a}]}{\mu(\s[a_{n-1}])}\right).
\end{align*}
Let $C^\ast_1:=\max\{M/\mu(\s[a]):a\in S^\ast\}$. Since $a_{n-1}\in S^\ast$,
$$
e^{\phi_n^\ast(\un{a},y)}= (C^\ast_1)^{\pm 1}\mu{[}\un{a}]\textrm{ for all }y\in \s[a_{n-1}].
$$
Since $[\un{a},\un{c}]\neq\emptyset$, $\s[a_{n-1}]\supseteq [\un{c}]$, so $\mu{[}\un{a},\un{c}]=\int e^{\phi_n^\ast(\un{a},y)}1_{[\un{c}]}(y)d\mu(y)= (C^\ast_1)^{\pm 1}\mu[\un{a}]\mu[\un{c}]$.

Now suppose $a_0\in S^\ast$ and  $[\un{c},\un{a}]\neq\emptyset$, where $\un{c}=(c_0,\ldots,c_{m-1})$. As before
$
\mu{[}\un{c},\un{a}]=\int_{[\un{a}]} L^m 1_{[\un{c}]}d\mu=\int_{[\un{a}]} e^{\phi_m^\ast(\un{c},y)}d\mu(y),
$
and $e^{\phi_m^\ast(\un{c},y)}=M^{\pm 1}(\frac{\mu[\un{c}]}{\mu^+(\s[c_{m-1}])})$. So
$$
\mu{[}\un{c},\un{a}]=\left(\frac{M^{\pm 1}}{\mu(\s[c_{m-1}])}\right)
\mu[\un{a}]\mu[\un{c}].
$$
Since $[\un{c},\un{a}]\neq\emptyset$, $\s[c_{m-1}]\supset [a_0]$, therefore the term in the brackets is in $[\frac{1}{M},\frac{M}{\mu[a_0]}]$. If we set $C^\ast_2:=\max\{M/\mu[a]:a\in S^\ast\}$, then
$\mu[\un{c},\un{a}]=(C^\ast_2)^{\pm 1}\mu[\un{a}]\mu[\un{c}]$.

The lemma follows with $C^\ast:=\max\{C^\ast_1,C^\ast_2\}$.
\end{proof}

\subsection{Equilibrium measures on two--sided CMS}\label{TwoSidedSection}
We return to two sided CMS $\Sigma=\Sigma(\mathfs G)$.  A function $\psi:\Sigma\to\R$ is called {\em weakly H\"older continuous} if there are constants $C>0$ and $0<\theta<1$ s.t. $\var_n\psi<C\theta^n$ for all $n\geq 2$, where
$
\var_n\psi:=\sup\{\psi(x)-\psi(y): x_i=y_i\ \ (i=-(n-1),\ldots,n-1)\}.
$

A function $\psi:\Sigma\to\R$ is called {\em one-sided}, if $\psi(\un{x})=\psi(\un{y})$ for every $\un{x},\un{y}\in\Sigma$ s.t. $x_i=y_i$ for all $i\geq 0$.
The following lemma was first proved (in a different setup) by Sinai. The proof given in \cite{BLNM} for subshifts of finite type also works for CMS:
\begin{lem}[Sinai]
If $\psi:\Sigma\to\R$ is weakly H\"older continuous and $\var_1\psi<\infty$, then there exists a bounded H\"older continuous function $\vf$ such that $\phi:=\psi+\vf-\vf\circ\s$ is weakly H\"older continuous and one--sided.
\end{lem}
\noindent
Notice that if $\psi$ is bounded then $\phi$ is bounded, and that every equilibrium measure for $\psi$  is  an equilibrium measure for $\phi$ and vice verse.

Since $\phi:\Sigma\to\R$ is one--sided, there is a function $\phi^+:\Sigma^+\to\R$ s.t. $\phi(\un{x})=\phi^+(x_0,x_1,\ldots)$. If $\phi:\Sigma\to\R$ is weakly H\"older continuous, then $\phi^+:\Sigma^+\to\R$ is weakly H\"older continuous.

Any shift invariant probability measure $\mu$ on $\Sigma$ determines a shift  invariant probability measure $\mu^+$ on $\Sigma^+$ through the equations
$$
\mu^+[a_0,\ldots,a_{n-1}]:=\mu(_0[a_0,\ldots,a_{n-1}])
$$
(cf. (\ref{Cyl1}) and (\ref{Cyl2})).
The map $\mu\mapsto\mu^+$ is a bijection, and it  preserves ergodicity and entropy.
It follows that $\mu$ is an ergodic equilibrium measure for $\phi$ iff $\mu^+$ is an ergodic equilibrium measure for $\phi^+$.
\begin{cor}\label{CorBS}
Suppose $\s:\Sigma\to\Sigma$ is topologically mixing.  If $\psi:\Sigma\to\R$ is weakly H\"older continuous, $\sup\psi<\infty$, $\var_1\psi<\infty$, and $P_G(\psi)<\infty$ then $\psi$ has at most one equilibrium measure $\mu$. This measure is the natural extension of an equilibrium measure of a potential $\phi:\Sigma^+(\mathfs G)\to\R$ which satisfies  assumptions {\em (a),(b),(c)}.
\end{cor}




\subsection{The Bernoulli property}\label{BernoulliSection}
 The {\em Bernoulli scheme} with probability vector $p=(p_a)_{a\in S}$ is  $(S^\Z,\mathfs B(S^\Z),\mu_p,\s)$ where $\s$ is the left shift map and $\mu_p$ is given by $\mu_p(_m[a_m,\ldots,a_n])=p_{a_m}\cdots p_{a_n}$. If $(\Omega,\mathfs F,\mu,T)$ is measure theoretically isomorphic to a Bernoulli scheme, then we say that $(\Omega,\mathfs F,\mu,T)$ is a {\em Bernoulli automorphism}, and $\mu$ has the  {\em Bernoulli property}. In this section we prove:

\begin{theorem}\label{ThmSymbolic}
 Every equilibrium measure of a weakly H\"older continuous potential $\psi:\Sigma(\mathfs G)\to\R$ on a topologically mixing countable Markov shift s.t. $P_G(\psi)<\infty$ and $\sup\psi<\infty$ has the Bernoulli property.
\end{theorem}
\noindent
This was proved by Bowen \cite{Bowen} in the case when $\mathfs G$ is finite.  See \cite{Ratner}  and \cite{W} for generalizations to larger classes of potentials.

We need some facts from Ornstein Theory.
Suppose $\b=\{P_1,\ldots,P_N\}$ is a finite measurable partition for an invertible probability preserving map $(\Omega,\mathfs F,\mu,T)$.  For every $m,n\in\Z$ s.t. $m<n$, let
$
\b_{m}^n:=\bigvee_{i=m}^n T^{-i}\b$.

\begin{defi}[Ornstein]
A finite measurable partition $\b$ is called {\em weak Bernoulli} if $\forall\e>0$ $\exists k>1$ s.t.
$
\sum\limits_{A\in\b_{-n}^0}\sum\limits_{B\in \b_k^{k+n}}|\mu(A\cap B)-\mu(A)\mu(B)|<\e\textrm{ for all }n>0.
$
\end{defi}
\noindent
Ornstein showed that if an invertible probability preserving transformation has a generating increasing sequence of weak Bernoulli partitions, then it is measure theoretically isomorphic to a Bernoulli scheme \cite{O1,OF}.

\medskip
\noindent
{\em Proof of Theorem \ref{ThmSymbolic}.} First we make a reduction to the case when  $\var_1\psi<\infty$. To do this, recode $\Sigma(\mathfs G)$ using the Markov partition of cylinders  of length two and notice that $\var_1$ of the  new coding equals $\var_2$ of the original coding.
The supremum and the pressure of $\psi$ remain finite, and the variations of $\psi$  continue to decay exponentially.

Suppose $\mu$ is an equilibrium measure of  $\psi:\Sigma(\mathfs G)\to\R$.
For every $\mathfs V'\subset \mathfs V$ finite, let
$
\a(\mathfs V'):=\bigl\{_0[v]:v\in\mathfs V'
\bigr\}\cup\bigl\{\bigcup_{v\not\in\mathfs V'} {_0[v]}\bigr\}.
$
We claim that $\a(\mathfs V')$ is weak Bernoulli. This implies  the Bernoulli property, because of the results of Ornstein we cited above.

 We saw in the previous section that the measure $\mu^+$ on $\Sigma^+(\mathfs G)$ given by
$$
\mu^+[a_0,\ldots,a_{n-1}]:=\mu({_0[}a_0,\ldots,a_{n-1}])
$$
satisfies $L^\ast\mu^+=\mu^+$ and $L^n 1_{[\un{a}]}\xrightarrow[n\to\infty]{}\mu^+[\un{a}]$, where $L=L_{\phi^\ast}$ and $\phi^\ast:\Sigma^+(\mathfs G)\to\R$ is weakly H\"older continuous. By (\ref{var}),
$$
\sup_{n\geq 1}(\var_{n+m}\phi_n^\ast)\xrightarrow[m\to\infty]{}0.
$$

Fix $0<\d_0<1$ so small that
$1-e^{-t}\in(\frac{1}{2}t,t)$ for all $0<t<\d_0$. Fix some smaller  $0<\delta<\d_0$, to be determined later, and
 choose
\begin{itemize}
\item a finite collection  $S^\ast$ of states (vertices) s.t. $\mu\bigl(\bigcup_{a\in S^\ast}{_0[}a]\bigr)>1-\d$;
\item a constant $C^\ast=C^\ast(S^\ast)>1$ as in Corollary \ref{CorProd};
\item a natural number $m=m(\d)$ s.t. $\sup_{n\geq 1}(\var_{n+m}\phi_n^\ast)<\d$;
\item a finite collection $\gamma$ of $m$--cylinders ${_0}[\un{c}]$ s.t. $\mu(\bigcup\g)>e^{-\d/2(C^\ast)^2}$;
\item points $x(\un{c})\in {_0[\un{c}]}\in\g$;
\item natural numbers $K(\un{c},\un{c}')$ $({[}\un{c}],{[}\un{c}']\in\g)$ s.t. for every $k\geq K(\un{c},\un{c}')$
$$
(L^k 1_{[\un{c}]})(x(\un{c}'))=e^{\pm \d}\mu^+[\un{c}]
$$
(recall that $L^n 1_{[\un{c}]}\xrightarrow[n\to\infty]{}\mu^+[\un{c}]$);
\item $K(\d):=\max\{K(\un{c},\un{c}'):{[\un{c}]},[\un{c}']\in\g\}+m$.
\end{itemize}

\medskip
\noindent
{\em Step 1.\/} Let $A:={_{-n}[}a_0,\ldots,a_{n}]$ and $B:={_k[}b_0,\ldots,b_{n}]$ be two non--empty cylinders of length $n+1$. If $b_0,a_{n}\in S^\ast$, then for every $k>K(\d)$ and every $n\geq 0$,
$$
|\mu(A\cap B)-\mu(A)\mu(B)|<2\sinh(10\d)\mu(A)\mu(B).
$$

\medskip
\noindent
{\em Proof.\/} Let $\a_m$ denote the collection of all $m$--cylinders $_0[\un{c}]$. For every $k>2m$,
\begin{align*}
\mu(A\cap B)&=\sum_{_0[\un{c}],_0[\un{c}'] \in\a_m}\mu(_{-n}[\un{a},\un{c}]\cap\s^{-(k-m)} {_0[}\un{c}',\un{b}])\\
&=\sum_{_0[\un{c}],{_0[}\un{c}']\in\a_m}\mu(_{0}[\un{a},\un{c}]\cap\s^{-(k+n-m)} {_0[}\un{c}',\un{b}])\ \ (\textrm{shift invariance})\\
&=\sum_{_0[\un{c}], {_0[}\un{c}']\in\a_m}\mu^+([\un{a},\un{c}]\cap\s^{-(k+n-m)} {[}\un{c}',\un{b}]).
\end{align*}
By Observation 2  in the proof of corollary \ref{CorProd},
\begin{align*}
\mu(A\cap B)&=\!\!\!\!\!\sum_{{_0[}\un{c}],{_0[}\un{c}']\in\g}\int_{[\un{c}',\un{b}]}(L^{k+n-m}1_{[\un{a},\un{c}]})d\mu^+
+\!\!\!\!\!\!\!\!\underset{{_0[}\un{c}]\not\in\g\textrm{ or }{_0[}\un{c}']\not\in\g}{\sum_{{_0[}\un{c}],{_0[}\un{c}']\in\a_m}}\!\!\int_{[\un{c}',\un{b}]}(L_\phi^{k+n-m}1_{[\un{a},\un{c}]})d\mu^+.
\end{align*}
We call the first sum the ``main term" and the second sum the ``error term".

To estimate these sums we use the following decomposition:
for every $y\in [\un{c}',\un{b}]$,
$(L^{k+n-m}1_{[\un{a},\un{c}]})(y)=\sum_{\s^{k-m-1}z=y}e^{\phi_{n+1}^\ast(\un{a},z)}e^{\phi_{k-m-1}^\ast(z)}1_{[\un{c}]}(z)$.

By the choice of $m$,
$
e^{\phi_{n+1}^\ast(\un{a},z)}=e^{\pm \d}e^{\phi_{n+1}^\ast(\un{a},w)}\textrm{ for all }w,z\in [\un{c}]$. Fixing $z$ and averaging over $w\in [\un{c}]$, we see that
\begin{align}
e^{\phi_{n+1}^\ast(\un{a},z)}&=e^{\pm\d}\left(\frac{1}{\mu^+[\un{c}]}\int_{[\un{c}]}e^{\phi_{n+1}^\ast(\un{a},w)}d\mu^+(w)\right)\notag\\
&=e^{\pm\d}
\left(\frac{1}{\mu^+[\un{c}]}\int(L^{n+1} 1_{[\un{a},\un{c}]})d\mu^+(w)\right)
=e^{\pm\d}\left(
\frac{\mu^+[\un{a},\un{c}]}{\mu^+[\un{c}]}
\right)\notag\\
\therefore\label{onegin}
(L^{n+k-m}1_{[\un{a},\un{c}]})(y)&=e^{\pm\d}\left(\frac{\mu^+[\un{a},\un{c}]}{\mu^+[\un{c}]}\right)(L_\phi^{k-m-1} 1_{[\un{c}]})(y)\textrm{ for $y\in [\un{c}',\un{b}]$.}
\end{align}

\medskip
\noindent
{\em Estimate of the main term\/}:
Suppose $k>K(\delta)$. If ${_0[}\un{c}],{_0[}\un{c}']\in\g$ and $y\in [\un{c}',\un{b}]$, then
\begin{align*}
(L^{k-m-1} 1_{[\un{c}]})(y)&=e^{\pm\d}(L^{k-m-1} 1_{[\un{c}]})(x(\un{c}'))\ \ \ \textrm{by choice of $m$ and since $y,x(\un{c}')\in[\un{c}']$}\\
&=e^{\pm 2\d}\mu^+[\un{c}]\textrm{ by choice of  $K(\d)$}.
\end{align*}
Plugging this into (\ref{onegin}), we see that  if $k>K(\d)$ then $(L_\phi^{n+k-m}1_{[\un{a},\un{c}]})(y)=e^{\pm 3\d}\mu^+[\un{a},\un{c}]$ on $[\un{c}',\un{b}]$. Integrating over $[\un{c}',\un{b}]$,
we see that for all $k>K(\d)$ the main term  equals
$$
e^{\pm 3\d}\sum_{{_0[}\un{c}],{_0[}\un{c}']\in\g}\mu^+[\un{a},\un{c}]\mu^+[\un{c}',\un{b}]=e^{\pm 3\d}\left(\sum_{{_0[}\un{c}]\in\g}\mu^+[\un{a},\un{c}]\right)
\left(\sum_{{_0[}\un{c}']\in\g}\mu^+[\un{c}',\un{b}]\right).
$$
The first bracketed sum is bounded above by $\mu^+[\un{a}]$. To bound it below, we use the  assumption that $a_{n}\in S^\ast$ to write
\begin{align*}
\sum_{{_0[}\un{c}]\in\g}\mu^+[\un{a},\un{c}]&=
\mu^+[\un{a}]-\sum_{{_0[}\un{c}]\in\a_m\setminus\g, [\un{a},\un{c}]\neq\emptyset}\mu^+[\un{a},\un{c}]\\
&=\mu^+[\un{a}]\left(1-
C^\ast\sum_{{_0[}\un{c}]\in\a_m\setminus\g,[\un{a},\un{c}]\neq\emptyset}\mu^+[\un{c}]
\right)\\
&\geq \mu^+[\un{a}]\left(1-
C^\ast\sum_{{_0[}\un{c}]\in\a_m\setminus\g}\mu^+[\un{c}]
\right)\\
&\geq \mu^+[\un{a}]\bigl(1-C^\ast (1-e^{-\d/2(C^\ast)^2})\bigr), \textrm{ by choice of $\g$}\\
&\geq e^{-\d}\mu^+[\un{a}],\textrm{ by choice of $\d_0$.}
\end{align*}
So the first bracketed sum is equal to  $e^{\pm\d}\mu^+[\un{a}]$. Similarly, the second bracketed sum is equal to $e^{\pm\d}\mu^+[\un{b}]$. Thus the main term is
$
e^{\pm 5\d}\mu^+[\un{a}]\mu^+[\un{b}]=e^{\pm 5\d}\mu(A)\mu(B).
$

\medskip
\noindent
{\em Estimate of the error term\/}: Since $a_{n}\in S^\ast$, (\ref{onegin}) implies that
$$
(L^{n+k-m}1_{[\un{a},\un{c}]})(y)\leq C^\ast e^\d \mu^+[\un{a}](L^{k-m-1} 1_{[\un{c}]})(y)\textrm{ on }[\un{c}',\un{b}].
$$
$$
\hspace{-2.6cm}\therefore\textrm{Error term}\leq
C^\ast e^\d \mu^+[\un{a}]\underset{{_0[}\un{c}]\not\in\g\textrm{ or }{_0[}\un{c}']\not\in\g}{\sum_{{_0[}\un{c}],{_0[}\un{c}']\in\a_m}}\int_{[\un{c}',\un{b}]}
(L^{k-m-1} 1_{[\un{c}]})(y)
d\mu^+
$$
\begin{align*}
&=
C^\ast e^\d \mu^+[\un{a}]\underset{{_0[}\un{c}]\not\in\g\textrm{ or }{_0[}\un{c}']\not\in\g}{\sum_{{_0[}\un{c}],{_0[}\un{c}']\in\a_m}}
\mu^+([\un{c}]\cap\s^{-(k-m-1)}[\un{c}',\un{b}])\\
&\leq C^\ast e^\d \mu^+[\un{a}]\biggl(\sum_{{_0[}\un{c}]\in\a_m\setminus\g}\mu^+([\un{c}]\cap\s^{-(k-1)}[\un{b}])+
\sum_{{_0[}\un{c}']\in\a_m\setminus\g}\mu^+(\s^{-(k-m-1)}[\un{c}',\un{b}])
\biggr)\\
&=C^\ast e^\d \mu^+[\un{a}]\biggl(\underset{{_0[}d_0,\ldots,d_{m-1}]\not\in\g}{\sum_{{_0[}\un{d}]\in\a_{k-1}}}\mu^+[\un{d},\un{b}]+
\sum_{{_0[}\un{c}']\in\a_m\setminus\g}\mu^+[\un{c}',\un{b}]
\biggr)\\
&\leq (C^\ast)^2e^\d\mu^+[\un{a}]\biggl(\underset{{_0[}d_0,\ldots,d_{m-1}]\not\in\g}{\sum_{{_0[}\un{d}]\in\a_{k-1}}}\mu^+[\un{d}]\mu^+[\un{b}]+
\sum_{{_0[}\un{c}']\in\a_m\setminus\g}\mu^+[\un{c}']\mu^+[\un{b}]
\biggr)\ \ (\because b_0\in S^\ast)\\
&\leq (C^\ast)^2e^\d\mu^+[\un{a}]\mu^+[\un{b}]\cdot 2\mu[(\cup\g)^c]\leq 2e^\d\d\mu^+[\un{a}]\mu^+[\un{b}]<5\d \mu^+[\un{a}]\mu^+[\un{b}].
\end{align*}
We get that the error term is less than $5\d \mu(A)\mu(B)$.

\medskip
We see that for all $k>K(\d)$,  $\mu(A\cap B)=(e^{\pm 5\d}\pm 5\d)\mu(A)\mu(B)$, whence
$$
|\mu(A\cap B)-\mu(A)\mu(B)|\leq \mu(A)\mu(B)\max\{e^{5\d}+5\d-1, 1-e^{-5\d}+5\d\}.
$$
It follows that $
|\mu(A\cap B)-\mu(A)\mu(B)|\leq 2\sinh(10\d)\mu(A)\mu(B)$.

\medskip
\noindent
{\em Step 2.\/} For every $k>K(\d)$, for every $n\geq 0$,
$$
\sum_{A\in\a_{-n}^0, B\in\a_k^{k+n}}|\mu(A\cap B)-\mu(A)\mu(B)|<2\sinh(10\d)+4\d.
$$

\medskip
\noindent
{\em Proof.\/} Write $A={_{-n}[}a_0,\ldots,a_n]$ and $B={_k[}b_0,\ldots,b_n]$. We break the sum into
\begin{enumerate}
\item the sum over $A,B$ s.t. $a_n, b_0\in S^\ast$;
\item the sum over $A,B$ s.t. $a_n\not\in S^\ast$;
\item the sum over $A,B$ s.t.  $a_n\in S^\ast$ and $b_0\not\in S^\ast$.
\end{enumerate}
The first sum is less than $2\sinh(10\d)$. The second and third sums are bounded by
$
2\mu[(\bigcup_{a\in S^\ast} {_0[}a])^c]<2(1-e^{-\d})<2\d.
$

\medskip
\noindent
{\em Step 3.\/} $\a(\mathfs V')$ has the weak Bernoulli property for every finite $\mathfs V'\subset \mathfs V$.

\medskip
\noindent
{\em Proof.\/} Choose $\d$ so small that $2\sinh(10\d)+4\d<\e$ and take $K=K(\d)$ as above, then
$
\sum_{A\in\a_{-n}^0, B\in\a_k^{k+n}}|\mu(A\cap B)-\mu(A)\mu(B)|<\e\textrm{ for all }n\geq 0.
$
Since the partitions $\a(\mathfs V')_{-n}^0$ and $\a(\mathfs V')_k^{k+n}$ are coarser  than $\a_{-n}^0$ and $\a_k^{k+n}$,  the weak Bernoulli property for $\a(\mathfs V')$  follows  by the triangle inequality.    \hfill$\Box$
\section{Step 3: The Non-Mixing Case}
\begin{lem}[Adler, Shields, and Smorodinsky]\label{LemmaRot}
Let $(X,\mathfs B,\mu,T)$ be an ergodic invertible probability preserving transformation with a measurable set $X_0$ of positive measure such that
\begin{enumerate}
\item $T^p(X_0)=X_0\mod\mu$;
\item $X_0, T(X_0), \ldots,T^{p-1}(X_0)$ are pairwise disjoint $\mod\mu$;
\item $T^p:X_0\to X_0$ equipped with $\mu(\cdot|X_0)$ is a Bernoulli automorphism.
\end{enumerate}
Then $(X,\mathfs B,\mu,T)$ is measure theoretically isomorphic to the product of a Bernoulli scheme and a finite rotation.
\end{lem}
\noindent
{\em Proof (see \cite{ASS})}.
Let $X_i:=T^i(X)$ ($i=0,\ldots,p-1$).
 Since $T$ is ergodic and measure preserving, $\mu(X_i)=\frac{1}{p}$ for all $p$. Also,  $T^p(X_i)=X_i\mod\mu$ for all $i$.
Since $T$ is invertible, $T^p:X_i\to X_i$ equipped with $\mu_i:=\mu(\cdot|X_i)$ is  isomorphic to $T^p:X_0\to X_0$. It follows that $h_{\mu_i}(T^p)$ are all equal. Since $\mu=\frac{1}{p}(\mu_0+\cdots+\mu_{p-1})$ and since $\mu\mapsto h_\mu(T^p)$ is affine, $h_{\mu_i}(T^p|_{X_i})=h_\mu(T^p)=p h_\mu(T)$ for every $i$.

Let $(\Sigma,\mathfs F,m,S)$ denote a Bernoulli scheme s.t.
$h_m(S)=h_\mu(T)$. The map $S^p:\Sigma\to\Sigma$ is isomorphic to a Bernoulli scheme with entropy $p h_\mu(T)$. It follows that $S^p$ is isomorphic to $T^p:X_0\to X_0$. Let $\vartheta:X_0\to\Sigma$ be an isomorphism map: $\vartheta\circ T^p=S^p\circ\vartheta$.
Define:
\begin{itemize}
\item $F_p:=\{0,1,\ldots,p-1\}$
\item $R:F_p\to F_p$, $R(x)=x+1\  (\mod p\,)$
\item $\Pi:X\to \Sigma\x F_p$, $\Pi(x)=(S^i[\vartheta(y)],i)$ for the unique $(y,i)\in X_0\x F_p$ s.t. $x=T^i(y)$ (this makes sense  on a set of full measure).
\end{itemize}

$\Pi$ is an isomorphism from $(X,\mathfs B,\mu,T)$ to $(\Sigma\x F_p, \mathfs F\otimes 2^{F_0},m\times c,S\x R)$, where $c$ is $\frac{1}{p}\times$the counting measure on $F_p$:
\begin{enumerate}
\item $\Pi$ is invertible: The  inverse function is $(z,i)\mapsto T^i(\vartheta^{-1}[S^{-i}(z)])$.
\item $\Pi\circ T=(S\x R)\circ\Pi$: Suppose $x\in X$, and write $x=T^i(y)$ with $(y,i)\in X_0\x F_p$. If $i<p-1$, then $T(x)=T^{i+1}(y)$ with $(y,i+1)\in X_0\x F_p$, so
$
\Pi[T(x)]=(S^{i+1}[\vartheta(y)],i+1)=(S\x R)(S^i[\vartheta(y)], i))=(S\x R)[\Pi(x)].
$
If $i=p-1$, then $T(x)=T[T^{p-1}(y)]$ and $(T^p(y),0)\in X_0\x F_p$. Since $\vartheta\circ T^p=S^p\circ\vartheta$ on $X_0$,
$
\Pi[T(x)]=(\vartheta(T^p y),0)=(S^p[\vartheta(y)],R(p-1))=(S\x R)(S^i[\vartheta(y)],i)=(S\x R)[\Pi(x)].
$
In all cases, $\Pi\circ T=(S\x R)\circ\Pi$.
\item $\mu\circ \Pi^{-1}=m\x c$: For every Borel set $E\subset \Sigma$ and $i\in F_p$,
\begin{multline*}
(\mu\circ \Pi^{-1})(E\x\{i\})=\mu[\vartheta^{-1} S^{-i}(E)]=\mu(X_0)\mu(\vartheta^{-1} S^{-i}(E)|X_0)\\
= \mu(X_0)(\mu_0\circ\vartheta^{-1})(S^{-i}E)=\frac{1}{p} m(S^{-i} E)=\frac{1}{p} m(E)
=(m\x c)(E\x\{i\}).
\end{multline*}
\end{enumerate}
It follows that $\Pi$ is a measure theoretic isomorphism.\hfill$\Box$

\subsection*{Proof of Theorem \ref{ThmMain}} Suppose $\mu$ is an equilibrium measure with positive entropy for $f$ and the H\"older potential $\Psi:M\to\R$. Fix some $0<\chi<h_\mu(f)$. By Theorems \ref{Theorem_Main_Extension} and \ref{Theorem_Main_Lift}, there exists a  countable Markov shift $\s:\Sigma\to\Sigma$,  a H\"older continuous map $\pi:\Sigma\to M$, and  a shift invariant ergodic probability measure $\wh{\mu}$ on $\Sigma$ s.t. $\wh{\mu}\circ\pi^{-1}=\mu$ and $h_{\wh{\mu}}(\s)=h_\mu(f)$.
In particular, if $\psi:=\Psi\circ\pi$, then
$
h_{\wh{\mu}}(\s)+\int\psi d\wh{\mu}=h_{{\mu}}(f)+\int{\Psi}d{\mu}.
$

For any other ergodic shift invariant probability measure $\wh{m}$, there is a set of full measure $\wh{\Sigma}\subset \Sigma$ s.t. $\pi:\wh{\Sigma}\to M$ is finite-to-one (Theorem \ref{Theorem_Main_Finite_To_One}). Therefore the $f$--invariant measure $m:=\wh{m}\circ\pi^{-1}$  has the same entropy as $\wh{m}$, whence
$$
h_{\wh{m}}(\s)+\int\psi d\wh{m}=h_{{m}}(f)+\int\Psi d{m}\leq h_{{\mu}}(f)+\int{\Psi}d{\mu}=h_{\wh{\mu}}(\s)+\int\psi d\wh{\mu}.
$$
It follows that $\wh{\mu}$ is an equilibrium measure for $\s:\Sigma\to\Sigma$ and $\psi$.

We wish to apply Theorem \ref{ThmSymbolic}. The potential $\psi$  is H\"older continuous, bounded, and $P_G(\psi)=h_\mu(f)+\int\Psi d\mu<\infty$.  But $\s:\Sigma\to\Sigma$ may  not be  topologically mixing.
To deal with this difficulty we appeal to the spectral decomposition theorem.

Since $\wh{\mu}$ is ergodic, it is carried by a topologically transitive $\Sigma'=\Sigma(\mathfs G')$ where $\mathfs G'$ is a subgraph of $\mathfs G$. Let $p$ denote the period of $\Sigma'$ (see \S \ref{SymbolicSection}).
The Spectral Decomposition Theorem for CMS  \cite[Remark 7.1.35]{Ki} states that
$$
\Sigma'=\Sigma'_0\uplus\Sigma_1'\uplus\cdots\uplus \Sigma_{p-1}'
$$
where every $\Sigma'_i$ is a union of states of $\Sigma$,  $\s(\Sigma_i')=\Sigma_{(i+1)\mathrm{mod} p}'$, and $\s^p:\Sigma_i'\to\Sigma_i'$ is  topologically mixing. Each $\s^p:\Sigma_i'\to\Sigma_i'$ is topologically conjugate to the CMS $\Sigma(\mathfs G'_i)$ where $\mathfs G'_i$ is the directed graph with
\begin{itemize}
\item vertices
$
(v_0,v_1,\ldots,v_{p-1})$
where $v_0\to\cdots\to v_{p-1}$ is a path in $\mathfs G'$ which starts at one of the states in $\Sigma'_i$, 
\item and edges
$
(v_0,\ldots,v_{p-1})\to(w_0,\ldots,w_{p-1})\textrm{ iff }v_{p-1}= w_0$.
\end{itemize}

Let $\wh{\mu}_i:=\wh{\mu}(\cdot|\Sigma_i')$. It is not difficult to see that $\wh{\mu}_i$ is an equilibrium measure for $\s^p:\Sigma_i'\to\Sigma_i'$ with respect to the potential $\psi_p:=\psi+\psi\circ\s+\cdots+\psi\circ\s^{p-1}$. It is also not difficult to see that $\psi_p$ can be identified with a bounded H\"older continuous potential ${\psi}_p^i$ on $\Sigma(\mathfs G'_i)$ and that $P_G(\psi_p^i)=pP_G(\psi)<\infty$.

By Theorem \ref{ThmSymbolic}, $\s^p:\Sigma_i'\to\Sigma_i'$ equipped with $\wh{\mu}_i$ is isomorphic to a Bernoulli scheme. 

Let $X_i:=\pi(\Sigma_i')$. Since $\pi\circ\s=f\circ\pi$, $f(X_i)=X_{(i+1)\mathrm{mod\ }p}$. Each $X_i$ is $f^p$--invariant, and $f^p:X_i\to X_i$ equipped with $\mu_i:=\mu(\cdot|X_i)$ is a factor of $\s^p:\Sigma_i'\to\Sigma_i'$. By Ornstein's Theorem \cite{O1}, factors of Bernoulli automorphisms are Bernoulli automorphisms. So $f^p:X_i\to X_i$ are Bernoulli automorphisms.

In particular, $f^p:X_i\to X_i$ are ergodic. Since $X_i\cap X_j$ is $f^p$--invariant, either $X_i=X_j$ or $X_i\cap X_j=\emptyset\mod \mu$. So there exists $q|p$ s.t.
$
M=X_0\uplus\cdots\uplus X_{q-1}\mod\mu.
$
Since $q|p$, $f(X_i)=X_{(i+1)\mathrm{mod\ }q}$, and 
 $f^q:X_0\to X_0$ is a root of $f^p:X_0\to X_0$. Since $f^p$ is Bernoulli, $f^q$ is Bernoulli  \cite{O3}. By Lemma \ref{LemmaRot},   $(M,\mathfs B(M),\mu,f)$ is isomorphic to the product of a  Bernoulli scheme and a finite rotation.\hfill$\Box$

\section{Concluding remarks}
We discuss some additional consequences of the proof we presented in the previous sections. In what follows $f:M\to M$ is a $C^{1+\a}$ surface diffeomorphism on a compact smooth orientable surface. We assume throughout that the topological entropy of $f$ is positive.

\subsection{The measure of maximal entropy is virtually Markov}\label{MaxSubSection}
Equilibrium measures for $\Psi\equiv 0$ are called {\em measures of maximal entropy} for obvious reasons.

A famous theorem of Adler \& Weiss \cite{AW} says that an ergodic measure of maximal entropy $\mu_{\max}$ for a hyperbolic toral automorphism  $f:\T^2\to \T^2$ can be coded as  finite state Markov chain.
More precisely, there exists a subshift of finite type $\s:\Sigma\to\Sigma$ and a H\"older continuous map $\pi:\Sigma\to\T^2$   such that (a) $\pi\circ\s=f\circ\pi$; (b)  $\mu_{\max}=\wh{\mu}_{\max}\circ\pi^{-1}$ where $\wh{\mu}_{\max}$ is an ergodic Markov measure on $\Sigma$; and (c) $\pi$ is a measure theoretic isomorphism.

 This was  extended by Bowen \cite{BLNM} to all Axiom A diffeomorphisms, using  Parry's characterization of the measure of maximal entropy for a subshift of finite type \cite{Pa}. Bowen's result holds in any dimension.

In dimension two, we have the following  generalization to  general $C^{1+\a}$ surface  diffeomorphisms with positive topological entropy:

\begin{theorem}\label{ThmMarkov}
Suppose $\mu_{\max}$ is an ergodic measure of maximal entropy for $f$, then there exists a topologically transitive CMS $\s:\Sigma\to\Sigma$ and a H\"older continuous map $\pi:\Sigma\to M$ s.t. {\em (a)} $\pi\circ\s=f\circ\pi$; {\em (b)} $\mu_{\max}=\wh{\mu}_{\max}\circ\pi^{-1}$ where $\wh{\mu}_{\max}$ is an ergodic  Markov measure on $\Sigma$; and {\em (c)} $\exists\Sigma'\subset\Sigma$ of full measure s.t.  $\pi|_{\Sigma'}$ is $n$--to--one.
\end{theorem}
\begin{proof}
The arguments in the previous section show that $\mu_{\max}=\wh{\mu}_{\max}\circ\pi^{-1}$ where $\wh{\mu}_{\max}$ is an ergodic measure of maximal entropy on some topologically transitive countable Markov shift $\Sigma(\mathfs G)$ and $\pi:\Sigma(\mathfs G)\to M$ is H\"older continuous map s.t. $\pi\circ\s=f\circ\pi$ and such that $\pi$ is finite-to-one on a set of full $\wh{\mu}_{\max}$--measure. Since $x\mapsto |\pi^{-1}(x)|$ is $f$--invariant, $\pi$ is $n$--to--one on a set of full measure for some $n\in\N$.

Gurevich's Theorem \cite{G} says that $\wh{\mu}_{\max}$ is  a Markov measure. Ergodicity forces the support of $\wh{\mu}_{\max}$ to be a  topologically transitive sub--CMS of $\Sigma(\mathfs G)$.
\end{proof}
\noindent
The example mentioned in the introduction shows that the theorem is false in dimension larger than two.

\subsection{Equilibrium measures for $-t\log J_u$}\label{SectionJ_u}
Theorem \ref{ThmMain} was stated for equilibrium measures $\mu$ of  H\"older continuous functions $\Psi:M\to\R$, but the proof  works equally well for any function $\Psi$ s.t. $\psi:=\Psi\circ\pi_\chi$ is a bounded H\"older continuous function on $\Sigma_\chi$. Here $\chi$ is any positive number strictly smaller than $h_\mu(f)$, and  $\pi_\chi:\Sigma_\chi\to M$ is the Markov extension described in \S\ref{SymbolicSection}.

We discuss a particular example which  appears naturally in hyperbolic dynamics (see e.g. \cite{BP}, \cite{Ledrappier},\cite{Bowen}).

 Let $M'$ denote the set of $x\in M$ s.t.  $T_x M$ splits into the direct sum of two one--dimensional spaces
$
E^s(x)$ and $E^u(x)$
so that $\limsup\limits_{n\to\infty}\frac{1}{n}\log\|df^n_x\un{v}\|_{f^n(x)}<0$ for all $\un{v}\in E^s(x)\setminus\{\un{0}\}$, and $\limsup\limits_{n\to\infty}\frac{1}{n}\log\|df^{-n}_x\un{v}\|_{f^{-n}(x)}<0$ for all $\un{v}\in E^u(x)\setminus\{\un{0}\}$.  It is well--known that if the spaces $E^s(x)$, $E^u(x)$ exist, then they are unique, and $df_x[E^u(x)]=E^u(f(x))$, $df_x[E^s(x)]=E^s(f(x))$.

\begin{definition}
The {\em unstable Jacobian} is
$
J_u(x):=|\det(df_x|_{E^u(x)})|
$  $(x\in M')$.
\end{definition}
\noindent
Equivalently, $J_u(x)$ is the unique positive number s.t. $\|df_x(\un{v})\|_{f(x)}=J_u(x)\|\un{v}\|_x$ for all $\un{v}\in E^u(x)$.

Notice that $J_u(x)$ is only defined on $M'$.
Oseledets' Theorem and Ruelle's Entropy Inequality guarantee that $\mu(M\setminus M')=0$ for every $f$--ergodic invariant measure with positive entropy.

The maps $x\mapsto E^u(x)$, $x\mapsto E^s(x)$ are in general  not smooth. Brin's Theorem states that these maps are H\"older continuous on Pesin sets \cite[\S5.3]{BP}. Therefore $J_u(x)$ is H\"older continuous on Pesin sets. We have no reason to  expect  $J_u(x)$ to extend to a  H\"older continuous function  on $M$.

Luckily, the following holds \cite[Proposition 12.2.1]{S}: For the Markov extension $\pi_\chi:\Sigma_\chi\to M$,  $E^u(\pi(\un{u})), E^s(\pi(\un{u}))$ are well--defined for every $\un{u}\in\Sigma$, and the maps
 $\un{u}\mapsto E^u(\un{u}), \un{u}\mapsto E^s(\un{u})$ are H\"older continuous on $\Sigma_\chi$.  As a result
$J_u\circ\pi$ is a globally defined bounded H\"older continuous function on $\Sigma_\chi$.

Since $f$ is a diffeomorphism, $\log (J_u\circ\pi)$ is also globally defined,  bounded and H\"older continuous.

\begin{theorem}
Suppose $\mu$ maximizes
$
h_\mu(f)-t\int(\log J_u) d\mu
$
among all ergodic invariant probability measures carried by $M'$. If
$h_\mu(f)>0$, then $f$ is measure theoretically isomorphic w.r.t. $\mu$ a Bernoulli scheme times  a finite rotation.
\end{theorem}
\noindent
The case $t=1$ follows from the work of  Ledrappier \cite{Ledrappier}, see also Pesin \cite{Pesin}.

\subsection{How many ergodic equilibrium measures with positive entropy?}

\begin{theorem}\label{ErgodicComponentsThm}
A H\"older continuous potential on $M$ has at most countably many ergodic equilibrium measures with positive entropy.
\end{theorem}
\begin{proof}
Fix $\Psi:M\to\R$ H\"older continuous (more generally a function such that $\psi$ defined below is H\"older continuous).

Given  $0<\chi<h_{top}(f)$, we show that $\Psi$ has at most countably many ergodic equilibrium measures $\mu$ s.t.  $h_\mu(f)>\chi$.

Let $\pi_\chi:\Sigma_\chi\to M$ denote the Markov extension described in \S\ref{SymbolicSection}, and let $\mathfs G$ denote the directed graph s.t. $\Sigma_\chi=\Sigma(\mathfs G)$.  We saw in the proof of Theorem \ref{ThmMain} that every ergodic equilibrium measure $\mu$ for $\Psi$ s.t. $h_\mu(f)>\chi$ is the projection of some ergodic equilibrium measure for $\psi:=\Psi\circ\pi_\chi:\Sigma(\mathfs G)\to\R$. So it is enough to show that $\psi$ has at most countably many ergodic equilibrium measures.

Every ergodic equilibrium measure $\mu$ on $\Sigma(\mathfs G)$ is carried by $\Sigma(\mathfs H)$ where (i) $\mathfs H$ is a subgraph of $\mathfs G$, (ii) $\s:\Sigma(\mathfs H)\to\Sigma(\mathfs H)$ is  topologically transitive, and (iii) $\Sigma(\mathfs H)$ carries  an equilibrium measure for $\psi:\Sigma(\mathfs G)\to\R$. Simply take the subgraph with vertices $a$ s.t. $\mu(_0[a])\neq 0$ and edges $a\to b$ s.t. $\mu(_0[a,b])\neq 0$.

For every subgraph $\mathfs H$ satisfying (i),(ii), and (iii) there is exactly one  equilibrium measure for $\psi$ on $\Sigma(\mathfs H)$. The support of this measure is $\Sigma(\mathfs H)$, see  Corollary \ref{CorBS} and Theorem \ref{ThmSymbolic}.

So every ergodic equilibrium measure sits on $\Sigma(\mathfs H)$ where $\mathfs H$ satisfies (i), (ii), and (iii), and every such $\Sigma(\mathfs H)$  carries exactly one  measure like that.
As a result, it is enough to show that $\mathfs G$ contains at most countably many subgraphs $\mathfs H$ satisfying (i), (ii), and (iii).

  We do this by showing that any two different subgraphs $\mathfs H_1$, $\mathfs H_2$ like that have disjoint sets of vertices.
 Assume by contradiction  that  $\mathfs H_1, \mathfs H_2$ share a vertex.  Then $\mathfs H:=\mathfs H_1\cup\mathfs H_2$ satisfies (i), (ii), and (iii). By the discussion above,  $\Sigma(\mathfs H)$  carries at most one equilibrium measure for $\psi$. But it carries at least two such measures: one with support $\Sigma(\mathfs H_1)$ and one with support $\Sigma(\mathfs H_2)$. This contradiction shows  that $\mathfs H_1$ and $\mathfs H_2$ cannot have common vertices.
\end{proof}

The case $\Psi=-\log J_u$ is due to Ledrappier \cite{Ledrappier} and Pesin \cite{Pesin}.  The case $\Psi\equiv 0$ was done at \cite{S}. Buzzi \cite{Buzzi} had shown that the measure of maximal entropy of a piecewise affine surface homeomorphism has finitely many ergodic components, and has conjectured that a similar result holds for $C^\infty$ surface diffeomorphisms with positive topological entropy.

\subsection{Acknowledgements} The author wishes to thank A. Katok and Y. Pesin for the suggestion to apply the results of \cite{S} to the study of the Bernoulli property of surface diffeomorphisms with respect to  measures of maximal entropy and equilibrium measures of $-t\log J_u$.

\end{document}